\documentclass[11pt,a4paper]{amsart}

\usepackage{amsmath}
\usepackage{amsthm}
\usepackage{enumerate} 
\usepackage{amssymb}
\usepackage{graphicx}
\usepackage[all]{xy}
\usepackage{hyperref}
\usepackage{aliascnt}
\usepackage{stmaryrd} 
\usepackage{mathtools}
\usepackage{verbatim} 
\usepackage{txfonts} 
\usepackage{lmodern}
\usepackage{setspace}   

\usepackage[hmargin=2.9cm]{geometry}
\newtheorem{lma}{Lemma}[section]

\newaliascnt{thmCt}{lma}
\newtheorem{thm}[thmCt]{Theorem}
\aliascntresetthe{thmCt}

\newaliascnt{corCt}{lma}
\newtheorem{cor}[corCt]{Corollary}
\aliascntresetthe{corCt}

\newaliascnt{propCt}{lma}
\newtheorem{prop}[propCt]{Proposition}
\aliascntresetthe{propCt}

\newtheorem*{thm*}{Theorem}
\newtheorem*{cor*}{Corollary}
\newtheorem*{prop*}{Proposition}

\theoremstyle{definition}

\newaliascnt{prgCt}{lma}
\newtheorem{prg}[prgCt]{}
\aliascntresetthe{prgCt}

\newtheorem*{prg*}{}

\newaliascnt{dfnCt}{lma}
\newtheorem{dfn}[dfnCt]{Definition}
\aliascntresetthe{dfnCt}

\newaliascnt{rmkCt}{lma}
\newtheorem{rmk}[rmkCt]{Remark}
\aliascntresetthe{rmkCt}

\newaliascnt{rmksCt}{lma}

\aliascntresetthe{rmksCt}

\newaliascnt{ntnCt}{lma}

\aliascntresetthe{ntnCt}

\newaliascnt{ntnsCt}{lma}

\aliascntresetthe{ntnsCt}

\newaliascnt{qstCt}{lma}

\aliascntresetthe{qstCt}

\newaliascnt{qstsCt}{lma}

\aliascntresetthe{qstsCt}

\newaliascnt{prblCt}{lma}

\aliascntresetthe{prblCt}

\newaliascnt{obsCt}{lma}

\aliascntresetthe{obsCt}

\newaliascnt{obssCt}{lma}

\aliascntresetthe{obssCt}

\newaliascnt{exaCt}{lma}

\aliascntresetthe{exaCt}

\newaliascnt{exasCt}{lma}

\aliascntresetthe{exasCt}

\newcommand{\CC}{\mathcal{C}}

\newcommand{\N}{\mathbb{N}}

\newcommand{\R}{\mathbb{R}}

\DeclareMathOperator{\EE}{E}

\DeclareMathOperator{\FF}{F}

\DeclareMathOperator{\LL}{L}
\DeclareMathOperator{\TT}{T}

\DeclareMathOperator{\Ped}{Ped} 
\DeclareMathOperator{\Prim}{Prim}

\newcommand{\Ca}{C^*}

\DeclareMathOperator{\Aff}{Aff}

\DeclareMathOperator{\Lat}{Lat}

\DeclareMathOperator{\Cu}{Cu}

\DeclareMathOperator{\AH}{AH}

\DeclareMathOperator{\AF}{AF}
\DeclareMathOperator{\Lsc}{Lsc}

\DeclareMathOperator{\ev}{ev}

\DeclareMathOperator{\id}{id}

\newcommand{\pprec}{\prec\hspace{-0,15cm}\prec}

\DeclareMathAlphabet{\mymathbb}{U}{bbold}{m}{n}

\begin{document}
\onehalfspacing
\title{Tracially reflexive $\Ca$-algebras}
\author{Laurent Cantier}

\address{Laurent Cantier\newline
Departamento de Matem\'aticas\\
Universidad de Zaragoza\\
C/Pedro Cerbuna 12\\
50009 Zaragoza\\
Spain}
\email[]{lncantier@gmail.com}

\thanks{The author was supported by the Spanish State Research Agency through Consolidaci\'on Investigadora program No. CNS2022-135340 and partially supported by MINECO (grants No. PID2023-147110NB-I00).}

\begin{abstract} Motivated by a question of L. Robert, asking whether $\LL(\TT(A))=\Lsc_\mathcal{C}(\TT(A))$ for any separable $\Ca$-algebra $A$, we introduce and initiate the study of \emph{tracially reflexive $\Ca$-algebras}. We first prove that commutative $\Ca$-algebras are tracially reflexive. We also prove that tracial reflexiveness satisfies permanence properties, such as being preserved under inductive limits. Subsequently, we expose two criteria for tracial reflexiveness, using the Cuntz semigroup and a weak version of the Schr\"{o}der-Simpson theorem, respectively. In particular, separable topological dimension zero $\Ca$-algebras are tracially reflexive. We end the manuscript by closing remarks that could lead to further lines of investigation involving tracial reflexiveness.
\end{abstract}

\maketitle
\section{Introduction}
The data encoded in tracial states of $\Ca$-algebras has been proven crucial to understand stably finite $\Ca$-algebras. It is no surprise that this information is part of the Elliott invariant and has played a central role in achieving the classification of (unital) separable, simple, $\mathcal{Z}$-stable $\Ca$-algebras satisfying the Universal Coefficient Theorem. See \cite{EGLN21} and \cite{GLN1,GLN2}. See also \cite{W18} for a general overview, and \cite{CGSTW21} for a remarkably detailed and innovative exposition on the matter. Indeed, these $\Ca$-algebras lie in stably finite ones. The study of \emph{dimension functions} and the \emph{Cuntz semigroup}, done in the pioneer works of Cuntz and Pedersen, proved the existence of tracial states for such $\Ca$-algebras and their strong link with equivalence classes of self-adjoint elements. See \cite{C78} and \cite{CP79}. To this day, the understanding of these tracial states remains key to the field, and hence, continues to be the focus of many works. We highlight a state-of-the art work contained in \cite{GM26}.

For more general $\Ca$-algebras, such as non-unital or non-simple, we know that tracial states information may be unsatisfactory, as their existence might be compromised. It is henceforth useful to look at more general objects, such as -unbounded- \emph{lower-semicontinuous traces}. The information contained in these maps can be hard to understand, especially in the non-simple case. The works of Robert, together with Elliott and Santiago, tend to generalize known facts about tracial states, in the broader context of traces. See \cite{ERS11, R09, R13}. They point out that the \emph{functionals} on the Cuntz semigroup of a $\Ca$-algebra encode the information about its traces. The set of these functionals is often seen as a \emph{dual} of the Cuntz semigroup. Naturally, the notion of \emph{double dual} has arised in the process. While it was unclear whether the canonical candidate would form a $\Cu$-semigroup, Robert has looked at a more adequate version, belonging to the category of $\Cu$-semigroups. He raises the question when these two notions coincide. See \cite[Question 1]{R09}. 
The main motivation of this manuscript lies in exploring $\Ca$-algebras having a positive answer to Robert's question, that we term \emph{tracially reflexive $\Ca$-algebras}. 

A starting point was found in the work \cite{K17} of Keimel, (re)making an explicit bridge between the Cuntz semigroup and the traces of a $\Ca$-algebra, via \emph{domain theory}. We mention that domain theory is a branch of computer science that studies order-theoretical topologies on monoids and (continuous) lattices, where the study of abstract $\Cu$-semigroups naturally fits. In the last part of \cite{K17}, Keimel highlights that traces on a $\Ca$-algebra can be seen as a domain-theoretical \emph{dual} of the monoid of positive elements equipped with an order-theoretic topology. Again, a notion of \emph{double dual} arises, and it is asked whether it always coincides with the \emph{round-ideal completion} of the (topological) monoid of positive elements. See \cite[Question 4.4]{K17}. With a bit of work, bridging $\Ca$-algebras and domain theory, it can be seen that the respective questions, independently raised by Robert and Keimel, are in fact equivalent. This equivalence allows us to use Keimel's work to answer Robert's question for commutative $\Ca$-algebras. We also obtain nice permanence properties.

\begin{thm*} Tracial reflexiveness is preserved under finite direct sums, passage to matrices, stabilization and inductive limits.
\end{thm*}

Subsequently, Cuntz semigroup techniques contained in \cite{MR23} allow us to deduce a first criterion for tracial reflexiveness, providing a proof that separable topological dimension zero $\Ca$-algebras are tracially reflexive. Lastly, a domain-theoretical version of the \emph{Schr\"{o}der-Simpson theorem} for \emph{dual pair of cones}, presented in \cite{K15}, yields a second abstract criterion for tracial reflexiveness. In the process, we study properties of unbounded lower-semicontinuous traces through their relationship with \emph{Pedersen ideals}. 

We gather the main results of our investigation in the following theorem.

\begin{thm*} The following classes of $\Ca$-algebras are tracially reflexive. 
\begin{itemize}
\item[(i)] Commutative $\Ca$-algebras and their inductive limits.
\item[(ii)] Separable $\Ca$-algebras such that $\Cu(A)\otimes \{0,\infty\}$ is algebraic. Equivalently, separable $\Ca$-algebras with topological dimension zero.
\item[(iii)] $\Ca$-algebra such that $\Ped(A)_+\cap I\subseteq \Ped(I)_+$, for any $I\in \Lat(A)$.
\end{itemize}
\end{thm*}

Let us point out that the nomenclature of \emph{tracial reflexiveness} comes from the domain-theoretical formulation of the problem. 

We end the manuscript with open lines of investigation and the following conjecture.
\begin{prg*}[\textbf{Conjecture}] Any stable rank one $\Ca$-algebra is tracially reflexive.\end{prg*}

\textbf{Acknowledgments}. The author would like to thank J. Bosa for the financial support, freedom and trust granted.\\

\textbf{Categories involved in the manuscript}. By a \textbf{cone}, we  mean an abelian monoid $(C,+)$, with a scalar multiplication by positive real numbers compatible with the addition. We do not assume cones to be cancellative. (Equivalently, we do not assume cones to embed into an ordered vectorial space.) By a \emph{cone morphism}, we mean a monoid morphism compatible with the scalar multiplication, and we denote the category of cones by $\boldsymbol{\mathcal{C}}$. 

By a \textbf{topological cone}, we mean a cone equipped with a topology making the addition and scalar multiplication jointly-continuous.  By a \emph{toplogical cone morphism}, we mean a continuous cone morphism, and we denote the category of topological cones by $\boldsymbol{\mathrm{T}op_\mathcal{C}}$.

By a \textbf{d-monoid}, we mean an (abelian) ordered monoid $(C,+,\leq)$, such that its \emph{Scott topology} (whose closed sets are lower-sets closed under suprema of increasing nets, as long as they exist) is compatible with the addition, and turns $C$ into a \emph{complete} topological space. (Remark that in the separable setting, we may replace nets by increasing sequences.)
By a \emph{d-morphism}, we mean a Scott-continuous monoid morphism, and we denote the category of d-monoids by $\boldsymbol{\rm DoM}$. 

The \textbf{compact-containment} relation on a d-monoid $S$ is an auxiliary relation determined by the order as follows. For any $s,t\in S$, we say that $s$ is \emph{compactly-contained} in $t$, and we write $s\ll t$ if, for any increasing sequence $(t_n)_n$ such that $t\leq \sup t_n$, there exists $n\in\N$ such that $s\leq t_n$. The set of elements that are compactly-contained in some other element will be denoted by $S_\ll=\{s\in S\mid \exists t\in S, s\ll t\}$.

By a \textbf{$\boldsymbol{\Cu}$-semigroup}, we mean a (sequentially) continuous d-monoid whose compact-containment relation is compatible with the addition. By a \emph{$\Cu$-morphism}, we mean a d-monoid morphism compatible with the compact-containment relation, and we denote the category of $\Cu$-semigroups by $\boldsymbol{\Cu}$. For a detailed and broad survey on $\Cu$-semigroups, we refer the reader to \cite{GP24}.

By a \textbf{$\boldsymbol{\Cu}$-cone}, we mean a $\Cu$-semigroup $(S,+,\leq)$ with a scalar multiplication by positive real numbers compatible with the addition, the order, and suprema of increasing sequences. (Remark that we do not require the scalar multiplication to be compatible with the compact containment relation.) By a \emph{$\Cu$-cone morphism}, we mean a $\Cu$-morphism compatible with the scalar multiplication, and we denote the category of $\Cu$-cones by $\boldsymbol{\Cu_{\mathcal{C}}}$. 

\section{Tracially reflexive $\Ca$-algebras}
This section sets the notion of \emph{tracially reflexive $\Ca$-algebras}. These $\Ca$-algebras can be thought of a stable analogue of \emph{tracially ordered $\Ca$-algebras}, introduced in \cite{J25}, by taking into account unbounded traces. 
We start off with basic facts (probably known by experts) for the sake of nomenclature and clarity. 
\begin{prg}\textbf{Tracial states and Cuntz-Pedersen relation}.  For a $\Ca$-algebra $A$, we denote the set of tracial states by $\mathrm{T}_1(A)$. It is well-known that endowed with the weak*-topology, this set becomes a Choquet simplex. In particular, it is a Hausdorff compact convex set and a complete lattice with the Riesz interpolation property for the point-wise order. 

We denote the set of real-valued affine continuous functions over $\mathrm{T}_1(A)$ by $\mathrm{Aff T}_1(A)$, and we set $\mathrm{Aff T}_1(A)_+:=\{f \mid f(\tau)\geq 0,  \forall \, \tau \in \mathrm{T}_1(A) \}$ and $\mymathbb{1}\colon \tau\mapsto 1$. It is well known that the order given by $\mathrm{Aff T}_1(A)_+$ makes $(\mathrm{Aff T}_1(A), \mymathbb{1})$ into an order-unit space (i.e., an Archimedean ordered vectorial space with an order-unit) which is complete for the order-unit norm. (See \cite{A71}.) It can also be checked that the order-unit norm is equal to the supremum norm. As a consequence, the set of linear maps between any two of these Banach order-unit spaces is equipped with the topology induced by supremum norm over the unit ball, which turns out to be the same as the supremum norm of the unit positive ball. That is, for any $\phi\colon E\rightarrow F$, we define $\Vert \phi \Vert:= \sup_{0\leq f\leq \mymathbb{1}}\{\phi(f)\}$.

Finally, we recall that the assignments $\mathrm{T}_1\colon \Ca \rightarrow \rm Choq$ and $\mathrm{AffT}_1\colon \Ca \rightarrow \R\rm Ban_{\mymathbb{1}}$ are well-defined contravariant and covariant continuous functors respectively. These functors play a significant role as part of the original Elliott invariant and its refinements. 

Another picture of the tracial information of $\Ca$-algebras can be found in the pioneer work \cite{CP79} of Cuntz and Pedersen. They consider a quotient construction on self-adjoint elements, which gives an algebraic description of the functor $\mathrm{Aff T}_1$, under certain assumptions. 

More precisely, for any $\Ca$-algebra $A$, they consider the (sub)set $A_0:=\{a-b \mid a,b \in A_+, a\sim_{\rm CP} b\}$ of $A_{\rm sa}$, where $\sim_{\rm CP}$ is an equivalence relation on $A_+$ given by $a\sim_{\rm CP} b$, whenever $a=\sum x_nx_n^*$ and $b=\sum x_n^*x_n$, for some sequence $(x_n)_n$ in $A$. 

Let $q\colon A_{\rm sa}\rightarrow A_{\rm sa}/A_0$ be the quotient map. Denote $A^q:=A_{\rm sa}/A_0$ and $A^q_+:=q(A_+)$. It is easily seen that the evaluation map $\mathrm{ev}\colon A_{\rm sa}\rightarrow\mathrm{AffT}_1(A)$ factorizes through $A_{\rm sa}/A_0$. Actually, the evaluation map is an isometric isomorphism $\ev\colon A_{\rm sa}/A_0\simeq \mathrm{AffT}_1(A)$, where $A_{\rm sa}/A_0$ is equipped with the quotient topology coming from the norm. As a result, we can characterize $A_0=\{ h\in A_{\rm sa}\mid   \forall \, \tau \in \mathrm{T}_1(A),\tau(h)=0 \}$. (Note that $A_0$ can also be characterized as the kernel of the \emph{universal trace} $\mathrm{T}r\colon A\rightarrow A/\overline{[A,A]} $ restricted to $A_{\rm sa}$.) 

By \cite{A71}, we know that the order given by $A^q_+$ makes $A^q$ into an ordered vectorial space. However, $A^q$ may fail to be Archimedean. This is due to the fact that, even though $a\sim_{\rm CP} b$ implies $a-b \in A_0$, the converse may fail. This subtlety will be addressed more thoroughly in the subsequent section. 

As a consequence, $q(A_+)$ embeds in $\mathrm{Aff T}_1(A)_+$, via the evaluation map, but this map may fail to be exhaustive. Equivalently, we have that $q(A_+)\subseteq\{h+A_0, h \in A_{\rm sa} \mid \forall \, \tau \in \mathrm{T}_1(A), \tau(h)\geq 0  \}$, but the inclusion may be strict. This has led Jacelon to introduce the notion of \emph{tracially ordered} $\Ca$-algebras.

\begin{dfn}[{\cite[Definition 2.6]{J25}}]
We say that a $\Ca$-algebra $A$ is \emph{tracially ordered} if 
\begin{itemize} 
\item[(i)] $\mathrm{T}_1(A)\neq \emptyset$. 
\item[(ii)] $q(A_+)=\{h+A_0, h \in A_{\rm sa} \mid \tau(h)\geq 0,  \forall \, \tau \in \mathrm{T}_1(A) \}$.
\end{itemize}
\end{dfn}

It is shown therein that the ordered vectorial space $A^q$ of  any tracially ordered $\Ca$-algebra $A$ is Archimedean, and that in the unital case, $(A^q,q(1_A))$ becomes a complete order-unit space for the order-unit norm, which agrees with the quotient norm. As a consequence, we straightforwardly obtain the following characterization for the unital case.

\begin{prop}
Let $A$ be a unital $\Ca$-algebra. The following are equivalent
\begin{itemize}
\item[(i)] $A$ is tracially ordered.
\item[(ii)] $(A^q, q(A_+),[1]_A)\simeq (\mathrm{AffT}_1(A),\mathrm{AffT}_1(A)_+,\mymathbb{1})$ as complete order-unit spaces.
\end{itemize}
\end{prop}

Among tracially ordered $\Ca$-algebras sit the following classes: simple unital exact stably finite $\Ca$-algebras, unital commutative $\Ca$-algebras, dimension drop $\Ca$-algebras. (See \cite[Proposition 2.7]{J25}.) Nevertheless, we can already see that tracial orderedness does not behave well with regard to permanence properties. Indeed, going to either inductive limits, the non-unital setting, or the non-simple setting, the information given by tracial states may be unsatisfactory. In fact, there might not be any bounded tracial functional on the $\Ca$-algebra at hand. E.g., the $\Ca$-algebra $\mathcal{K}$ of compact operators over any infinite dimensional Hilbert space, has an empty tracial state space and hence, is not tracially ordered. As a consequence, it is not true that all $\AF$-algebras are tracially ordered, nor it is known if unital ones are. Similar questions arise for larger classes, such as $\AH$-algebras or inductive limits of dimension-drop $\Ca$-algebras.

To remedy this issue, the notion of \emph{traces} becomes the adequate generalization to consider. In view of obtaining a robust analogue of tracial orderedness, we work with traces instead of tracial states. Let us recall some facts that have been gathered around traces of $\Ca$-algebras. We refer the reader to \cite{ERS11,R09,R13} for more details.
\end{prg}

\begin{prg} \textbf{The cone of traces and the Robert relation}. By a \emph{trace} on a $\Ca$-algebra $A$, we mean a cone morphism $\tau\colon A_+\rightarrow [0,\infty]$ satisfying the trace identity, i.e. such that $\tau(xx^*)=\tau(x^*x)$ for any $x\in A$. 
By a \emph{lower-semicontinuous function on a topological space $X$}, we mean a continuous map $f\colon X\rightarrow [0,\infty]$, where $[0,\infty]$ is given \emph{the Scott topology}. We recall that this topology is characterized by its closed sets being the lower-sets closed under suprema of increasing sequences. Equivalently, any map $f\colon X\rightarrow [0,\infty]$ is lower-semicontinuous, if $f^{-1}((r,\infty])$ is open, for any $r\in \R_+$.

We denote the set of lower-semicontinuous traces on a $\Ca$-algebra $A$ by $\mathrm{T}(A)$. It forms an ordered monoid for the point-wise operations, which is a complete lattice. Following \cite{ERS11}, we consider a topology on $\TT(A)$, in which a given net $(\tau_i)_i$ converges to $\tau$, if \[\lim\sup \tau_i((a-\epsilon)_+)\leq \tau(a)\leq \lim\inf \tau_i(a)\] for any $a\in A_+$ and any $\epsilon>0$. (Here, $(a-\epsilon)_+$ denotes the \textquoteleft $\epsilon$ cut-down\textquoteright\ of $a$ obtained via functional calculus, after applying $f_\epsilon:=\max(0,\id-\epsilon)$.) We may refer to the latter as the \emph{interval topology}, that we denote $\tau_{iv}$. (See \autoref{prg:Keimel}.)

Endowed with the interval topology, $\TT(A)$ becomes topological cone (i.e., the addition and scalar multiplication are jointly continuous) which is compact Hausdorff and (locally) convex.  The category $\mathcal{C}_{\rm Choq}$ of \emph{extended Choquet cones} has been introduced in \cite{MR23}, to describe appropriately the categorical framework of the functor $\mathrm{T}\colon \Ca\rightarrow  \mathcal{C}_{\rm Choq}$. 
 
We have previously seen that the Cuntz-Pedersen equivalence relation $\sim_{\rm CP}$ on the set of positive elements is closely linked to the evaluation on tracial states. More concretely, we can consider a canonical subequivalence induced by $\sim_{\rm CP}$ and the order on $A_+$ as follows: $a\lesssim_{\rm CP}b$ whenever $a\sim_{\rm CP}a'\leq b$ for some $a'\in A_+$. Remark that $a\sim_{\rm CP}b$ if $a\lesssim_{\rm CP}b$ and $b\lesssim_{\rm CP}a$. The original work \cite{CP79} of Cuntz and Pedersen shows that $a\lesssim_{\rm CP}b$ implies $q(a)\leq q(b)$, which is in turn equivalent to $\tau(a)\leq\tau(b)$ for any $\tau\in \mathrm{T}_1(A)$. A fortiori, $a\sim_{\rm CP}b$ implies that $\tau(a)=\tau(b)$ for any $\tau\in \mathrm{T}_1(A)$.
Nevertheless, the converse need not be true. This is easily seen in any $\Ca$-algebras with an empty tracial state space, (e.g. $\mathcal{K}$), since the consequent is vacuously true. 

To address this subtlety, Robert has introduced in \cite{R09}, a stronger subequivalence relation on positive elements, related to comparison on traces. We write $a\lesssim_{\rm R}b$ if, for any $\epsilon >0$, there exists some $\delta>0$ such that $(a-\epsilon)_+ \lesssim_{\rm CP}  (b-\delta)_+$. Inverting the process of before, this subequivalence yields an equivalence relation on $A_+$ as follows: $a\sim_{\rm R}b$ whenever $a\lesssim_{\rm R}b$ and $b\lesssim_{\rm R}a$. Remark that our notations differ from the original paper \cite{R09} to avoid any confusion with the Cuntz-Pedersen (sub)equivalence.
The main result of \cite{R09} states the following.

\begin{thm}[{\cite[Theorem 1]{R09}}]\label{thm:R09} Let $A$ be a $\Ca$-algebra and $a,b\in A_+$. 
\begin{itemize}
\item[(i)] $a\lesssim_{\rm R} b $ if and only if  $\tau(a)\leq \tau(b)$ for any $\tau \in \mathrm{T}(A)$.
\item[(ii)] $a\sim_{\rm R} b $ if and only if $a\sim_{\rm CP} b$ if and only if $\tau(a)=\tau(b)$ for any $\tau \in \mathrm{T}(A)$.
\end{itemize}
\end{thm}

Let $A$ be a $\Ca$-algebra. Let us denote the set of lower-semicontinuous functions on $\TT(A)$ which are also cone morphisms, by $\Lsc_{\mathcal{C}}(\TT(A))$. (That is, the set of cone morphisms $f\colon\TT(A)\rightarrow [0,\infty]$ that are $(\tau_{iv},\tau_{\rm Scott})$-continuous.) Observe that $\Lsc_{\mathcal{C}}(\TT(A))$ is a d-cone (i.e., a d-monoid which is also a cone) for the point-wise operations. For any $a\in A_+$, it has been shown in \cite{ERS11} that the evaluation map $\overline{a}\colon \TT(A)\rightarrow [0,\infty]$ is an element of  $\Lsc_{\mathcal{C}}(\TT(A))$. Moreover, it is well-known that any lower-semicontinuous trace on $A$ extends (uniquely) to $A\otimes \mathcal{K}$. More particularly, we have that $\TT(A\otimes \mathcal{K})\simeq \TT(A)$ and for any $a\in (A\otimes\mathcal{K})_+$, the evaluation map $\overline{a}$ also belongs to $\Lsc_{\mathcal{C}}(\TT(A))$. 

Lastly, it is not known whether $\Lsc_{\mathcal{C}}(\TT(A))$ is a $\Cu$-semigroup. Nevertheless, it contains a $\Cu$-cone termed $\LL(\TT(A))$, consisting of functions $f\in\Lsc_{\mathcal{C}}(\TT(A))$ that can be expressed as suprema of increasing sequences $(f_n)_n$ such that $f_n$ is $\tau_{iv}$-continuous at each point where $f_{n+1}$ is finite. (See \cite{ERS11,R09}.)
\end{prg}

\begin{prg} \textbf{The functor $\mathcal{R}$}.
For a $\Ca$-algebra $A$, we equip $(A\otimes\mathcal{K})_+/\!\!\sim_{\rm R}$ with an addition, a $\R_+$-multiplication and an order as follows: $[a]_{\rm R}+[b]_{\rm R}:=[a\oplus b]_{\rm R}$, $r[a]_{\rm R}:=[ra]_{\rm R}$ and $[a]_{\rm R}\leq[b]_{\rm R}$ whenever $a\lesssim_{\rm R}b$. It can be checked that these are well-defined (e.g. from \autoref{thm:R09}) and turn $((A\otimes\mathcal{K})_+/\!\!\sim_{\rm R},+,\leq)$ into an ordered cone, that we denote $\mathcal{R}(A)$. It can also be checked that for any *-homomorphism $\phi\colon A\rightarrow B$ between $\Ca$-algebras, the map $\mathcal{R}(\phi)\colon \mathcal{R}(A)\rightarrow \mathcal{R}(B)$ sending $[a]_{\rm R}\rightarrow [\phi\otimes \id_{\mathcal{K}}(a)]_{\rm R}$ is a well-defined ordered cone morphism. From \autoref{thm:R09}, we see that the map $\iota\colon\mathcal{R}(A) \rightarrow  \Lsc_{\mathcal{C}}(\TT(A))$ sending $[a]_{\rm R}\mapsto\overline{a}$, is a well-defined order-embedding of cones, preserving suprema of increasing sequences and the compact-containment relation. Henceforth, we can identify elements of $\mathcal{R}(A)$ with elements of $ \Lsc_{\mathcal{C}}(\TT(A))$.

The following proposition, based on \cite[Remark 5.14]{ERS11}, shows that we get a well-defined continuous functor $\mathcal{R}\colon \Ca\rightarrow\Cu_{\mathcal{C}}$. 

\begin{prop}\label{prop:basisR}
Let $A$ be a $\Ca$-algebra. Then $\mathcal{R}(A)$ is a $\Cu$-cone and $(A_+/\!\!\sim_{\rm R},+,\leq)$ is $\ll$-dense in $\mathcal{R}(A)$.
\end{prop}

\begin{proof}
The explicit description $\LL(\TT(A))=\{\overline{a}\mid a\in (A\otimes\mathcal{K})_+\}$ given in \cite[Theorem 3.1]{R09}, together with \cite[Theorem 5.12]{ERS11} yield the desired result.
\end{proof}

A direct consequence of the above is the following lemma. 

\begin{cor}\label{cor:triangle} 
There is a functorial equivalence between $\mathcal{R}\simeq \LL(\TT(\_))$, where $\LL(\TT(\_))\colon \Ca\rightarrow \Cu_\mathcal{C}$ is the continuous functor of \cite[\S 5]{ERS11}. 

As a consequence, $\mathcal{R}\colon \Ca\rightarrow\Cu_{\mathcal{C}}$ is a well-defined continuous functor. 
\end{cor}
 \end{prg}

\begin{rmk}[{The $\triangleleft$ relation}]\label{rmk:triangle}
From the results of \cite{ERS11}, we have that $(A_+/\!\!\sim_{\rm R},+,\leq)$ is in fact $\triangleleft$-dense in $\mathcal{R}(A)$, where $\triangleleft$ is a strengthening of $\ll$, introduced in \cite{R13}, given by $f\triangleleft g$ in $\Lsc_{\mathcal{C}}(\TT(A))$, if $f\leq (1-\epsilon)g$ for some $\epsilon >0$ and $f$ is $\tau_{iv}$-continuous at each point where $g$ is finite.
\end{rmk}

\begin{prg}\textbf{Tracial Reflexiveness and its permanence properties}. In the general case, it is not clear whether the order-embedding $\iota\colon\mathcal{R}(A) \lhook\joinrel\rightarrow  \Lsc_{\mathcal{C}}(\TT(A))$ is a $\Cu_\mathcal{C}$-isomorphism. Even more so, it is not clear whether  $\Lsc_{\mathcal{C}}(\TT(A))$ is a $\Cu$-semigroup at all. (It is not known whether it satisfies axiom (O2).) This leads us to introduce the notion of \emph{tracially reflexive} $\Ca$-algebras as follows.

\begin{dfn}[{Tracial reflexiveness}]  
We say that a $\Ca$-algebra $A$ is \emph{tracially reflexive}, if $\iota\colon\mathcal{R}(A) \simeq  \Lsc_{\mathcal{C}}(\TT(A))$ is a $\Cu$-cone isomorphism. 

Equivalently, $\{\overline{a} \mid  a\in (A\otimes\mathcal{K})_+\}=\Lsc_{\mathcal{C}}(\TT(A))$. 
\end{dfn}

We now tackle permanence properties that tracial reflexiveness may have. We are particularly interested to see that tracial reflexiveness is preserved under inductive limits. 

\begin{thm} Tracial reflexiveness is preserved under passage to matrices, stabilization, (finite) direct sums and inductive limits.
\end{thm}

\begin{proof} Passage to matrices, stabilization and (finite) direct sums are straightforward and left to the reader. 

Let $(A_i,\phi_{i,j})_i$ be an inductive system of tracially reflexive $\Ca$-algebra and we denote its limit in $\Ca$ by $(A,\phi_{i,\infty})_i$. Let us show that $A$ is tracially reflexive. 

First, by the continuity of the functor $\TT\colon \Ca\rightarrow \mathcal{C}_{\rm Choq}$, provided by \cite[Theorem 3.12]{ERS11}, we observe that $\underset{\leftarrow}{\lim} (\TT(A_i),\TT(\phi_{i,j}))\simeq (\TT(A),\TT(\phi_{i,\infty}))_i$. We then apply the contravariant functor  $\Lsc_\mathcal{C}\colon \mathcal{C}_{\rm Choq}\rightarrow \rm DoM$ to the above. We obtain an inductive system and a cocone in the category $\rm DoM$. It can be checked that the cocone is in fact the $\rm DoM$-colimit. That is,  $\mathrm{DoM}-\underset{\rightarrow}{\lim} \Lsc_\mathcal{C}[(\TT(A_i),\TT(\phi_{i,j}))_i]\simeq\Lsc_\mathcal{C}[(\TT(A),\TT(\phi_{i,\infty}))_i]$. See e.g. \cite{F92}.

On the other hand, we apply the continuous functor $\mathcal{R}$ to the inductive system in $\Ca$, and we obtain a $\Cu_\mathcal{C}$-limit isomorphic to $\mathcal{R}(A)$. Since the connecting maps and the building blocks of the inductive systems are the same, we get an ordered monoid morphism $\theta\colon \Lsc_\mathcal{C}(\TT(A))\rightarrow \mathcal{R}(A)$. We also know that there is an order-embedding $\iota\colon\mathcal{R}(A) \lhook\joinrel\rightarrow  \Lsc_{\mathcal{C}}(\TT(A))$. It readily follows from construction that these maps are inverse of one another. The result is now obtained from \cite[Lemma 4.7]{C21a}.
\end{proof}
\end{prg}

\section{The commutative case}

This section is dedicated to prove that (inductive limits of) commutative $\Ca$-algebras are tracially reflexive. The proof relies on arguments exposed in Keimel's work (see \cite{K17}), where he studied the Cuntz semigroup and the cone of traces of a $\Ca$-algebra, through the lenses of \emph{domain theory}. 

We highlight that these arguments have inspired the nomenclature of such $\Ca$-algebras.

\begin{prg}\textbf{Domain Theory and Cuntz semigroups}.\label{prg:Keimel} It has been observed by Keimel and Thiel that many facts around the Cuntz semigroup of a $\Ca$-algebra, including facts on the cone of lower-semicontinuous traces and its compact Hausdorff topology $\tau_{iv}$ (which coincides with the cone of functionals on the Cuntz semigroup, as we will recall subsequently), could be directly deduced from well-known results of \emph{domain theory}. We refer the reader to \cite{K17} and \cite{T17} for the close relationship between domain theory and the theory of $\Cu$-semigroups, and \cite{GHKLMS03} for an overview on domain theory. Below, we summarize briefly but thoroughly, the main bridges built in \cite{K17}, on which we walk across to conclude that any commutative $\Ca$-algebra is tracially reflexive.

A \emph{preCuntz semigroup} is a monoid $(M,+)$, equipped with a transitive, interpolating, additive binary relation $\pprec$. This relation naturally induces a topology $\tau_{\pprec}$. A \emph{round-ideal of $M$} is a $\pprec$-downward hereditary, $\pprec$-directed subset of $M$. E.g., $m^{\pprec}:=\{n \in M \mid n\pprec m\}$ is a round-ideal, for any $m\in M$. The set of round-ideals $\mathcal{RJ}(M)$ of a preCuntz semigroup $(M,+,\pprec)$, equipped with the inclusion order and the addition given by $I+J:=\bigcup_{a\in I,b\in J}(a+b)^{\pprec}$ becomes a $\Cu$-semigroup. Furthermore, $\{m^{\pprec}\}$ is a basis of $\mathcal{RJ}(M)$ and the canonical map $M\rightarrow \mathcal{RJ}(M)$ sending $m\mapsto m^{\pprec}$ is a $(\tau_{\pprec},\tau_{\rm Scott})$-continuous monoid morphism. (See \cite[\S 3]{K17}.)

The \emph{dual $M^*$ of a preCuntz semigroup $M$} is the set of lower-semicontinuous functions on $M$ that are also monoid morphisms, i.e. monoid morphisms $M\rightarrow [0,\infty]$ that are $(\tau_{\pprec},\tau_{\rm Scott})$-continuous. For the point-wise operations, $(M^*,+,\leq)$ is a d-cone, i.e. a d-monoid which is also a cone. Moreover, we can equip $M$ either with the \emph{weak*-upper topology} $\tau_{up}$ (i.e., the coarsest topology rendering lower-semicontinuous all the evaluation maps) or the \emph{interval topology} $\tau_{iv}$, which is a refinement of the former. (See \cite[Proposition 3.7]{K17}.)

The \emph{double dual $M^{**}$ of a preCuntz semigroup $M$} is the set of lower-semicontinuous functions on $(M^*,\tau_{up})$ that are also cone morphisms, i.e. cone morphisms $M^*\rightarrow [0,\infty]$ that are $(\tau_{up},\tau_{\rm Scott})$-continuous. 

Back to our setting, the cone of traces $\TT(A)$ of any $\Ca$-algebra $A$ can be seen as the dual $(A_+,+,\pprec)^*$ of the preCuntz semigroup $(A_+,+,\pprec)$, where $a\pprec b$ if there exists $\delta>0$ such that $a\lesssim_{\rm CP}(b-\delta)_+$. (See \cite[\S 4]{K17}. We highlight that have used our notation $\lesssim_{\rm CP}$ for consistency, whereas Keimel has sticked with the notations of \cite{R09}.)
\end{prg}

We are now ready to prove the main result of this section, based on domain theoretical arguments. Let us start with a lemma regarding $\Cu$-semigroups.

\begin{lma} Let $S,T$ be $\Cu$-semigroups and let $\alpha\colon S\rightarrow T$. Assume that there exists a $\ll$-dense subset $B$ of $ S$ such that 
\begin{itemize}
\item[(i)] The restriction $\alpha_{\mid B}\colon B\rightarrow T$ is an order-embedding.
\item[(ii)] The image $\alpha(B)$ is $\ll$-dense in $T$
\end{itemize}

Then $\alpha$ is a $\Cu$-isomorphism.
\end{lma}

\begin{proof} Let $x,y\in S$ be such that $\alpha(x)\leq \alpha(y)$. Since $B$ is $\ll$-dense, we can find $\ll$-increasing sequences $(b_n)_n,(c_n)_n$ of elements of $B$ whose respective suprema are $x,y$. Applying $\alpha$, we obtain $\ll$-increasing sequences $\alpha(b_n)_n,\alpha(c_n)_n$ whose respective suprema are $\alpha(x),\alpha(y)$. Therefore, for any $n\in \N$, there exists $m\in \N$ such that $\alpha(b_n)\ll \alpha(b_m)$. Since $\alpha_{\mid B}$ is an order-embedding, we deduce that for any $n\in \N$, there exists $m\in \N$ such that $b_n\leq b_m$. Taking suprema on the right side first, and left side then, we conclude that $x\leq y$ and hence, $\alpha$ is an order-embedding. 

Next, take $t\in T$. There exists a $\ll$-increasing sequence $(b_n)_n$ of elements of $B$ such that $t=\sup\alpha(b_n)$. Write $s:=\sup b_n$. Since $\alpha$ preserves suprema of increasing sequences, we deduce that $\alpha(s)=t$ and hence, $\alpha$ is a $\Cu$-isomorphism.
\end{proof}

\begin{thm}\label{thm:commu} Let $A$ be a $\Ca$-algebra. 
There is an explicit $\Cu$-cone isomorphism $\mathfrak{r}\colon\mathcal{R}(A)\simeq \mathcal{RJ}(A)$ sending $[a]_{\rm R}\mapsto a^{\pprec}$, for any $a\in A_+$.
\end{thm}

\begin{proof}
Consider the canonical cone morphism $\pi\colon A_+\rightarrow \mathcal{R}(A)$ sending $a\mapsto [a]_{\rm R}$. Let us show that $\pi$ is a $(\tau_{\pprec},\tau_{\ll})$-continuous morphism preserving the preCuntz relations. (Here, $(\mathcal{R}(A),+,\ll)$ is considered as a preCuntz semigroup, and its induced topology $\tau_{\ll}$ is exactly the Scott topology.) 

By \cite[Proposition 5.3]{ERS11}, we know that $\overline{(a-\epsilon)}_+\ll \overline{a}$ for any $a\in (A\otimes \mathcal{K})_+$ and any $\epsilon >0$. Equivalently, $[(a-\epsilon)_+]_{\rm R}\ll [a]_{\rm R}$. On the other hand, it is easily seen that for $a,b\in A_+$ such that $a\pprec b$, there exists $\epsilon>0$ such that $a\pprec (b-\epsilon)_+\pprec b$. (Take any $\epsilon < \delta$.) We deduce that $\pi(a)\leq \pi(b-\epsilon)\ll \pi(b)$. Henceforth,  $\pi$ preserves the preCuntz relations.
Now, let $a\in A_+$ and $r\in \mathcal{R}(A)$ be such that $r\ll \pi(a)$. Again, by \cite[Proposition 5.3]{ERS11}, we can find some $\epsilon>0$ such that $r\ll\pi((a-\epsilon)_+)\ll \pi(a)$. Therefore, we have found some $b:=(a-\epsilon)_+$ such that $b\pprec a$ and such that $r\ll\pi (b)\ll \pi(a)$. We conclude that $\pi$ is $(\tau_{\pprec},\tau_{\ll})$-continuous, by the characterization given in \cite[\S 2.7]{K17}. 

As a consequence, we can invoke the universal property of the round-ideal completion (see \cite[Corollary 3.3]{K17}) stating that there exists a unique Scott-continuous monoid morphism $\mathfrak{r}\colon \mathcal{RJ}(A)\rightarrow \mathcal{R}(A)$ preserving the compact-containment relation, i.e. a $\Cu$-morphism, such that $\mathfrak{r}\circ \iota^{\pprec}=\pi$, where $\iota^{\pprec}\colon A_+\rightarrow \mathcal{RJ}(A)$ is the canonical morphism sending $a\mapsto a^{\pprec}$. 

Let us show that $\mathfrak{r}$ is an isomorphism, using our lemma. From construction, we know that the subset $\{a^{\pprec}\mid a\in A_+\}$ is $\ll$-dense in $\mathcal{RJ}(A)$ and by \autoref{prop:basisR}, we also know that the subset $\{[a]_{\rm R}\mid a\in A_+\}$ is $\ll$-dense in $\mathcal{R}(A)$. 
As explained in the discussion after \cite[Proposition 4.1]{K17}, for any $a,b\in A_+$, we have that $a^{\pprec}\subseteq b^{\pprec}$ if and only if $a\lesssim_{\rm R} b$. Together with the fact that $\mathfrak{r}\circ\iota^{\pprec}=\pi$, we obtain that $\mathfrak{r}$ and $\{a^{\pprec}\mid a\in A_+\}$ satisfy both conditions of the above lemma. 

Finally, it follows from construction that $\mathfrak{r}$ preserves scalar multiplication by positive real numbers.
We  conclude that $\mathfrak{r}\colon \mathcal{RJ}(A)\simeq \mathcal{R}(A)$ sending $a^{\pprec}\mapsto [a]_{\rm R}$ is a $\Cu$-cone isomorphism.
\end{proof}

\begin{thm}\label{thm:commutative}
Any commutative $\Ca$-algebra is tracially reflexive.
\end{thm}

\begin{proof}
First, observe that we have $\Lsc(M^*,\tau_{up})\subseteq \Lsc(M^*,\tau_{iv})$, for any preCuntz semigroup $M$, since the interval topology is finer than the weak*-upper topology. Moreover, we have a form of converse as follows.

\textit{Claim:} Any order-preserving lower-semicontinuous map on $(M^*,\tau_{up})$ is a lower-semi\-continuous map on $(M^*,\tau_{iv})$. 

This relies upon the fact that any subset of $M^*$ is $\tau_{up}$-open if and only if it is $\tau_{iv}$-open. (This is an analogous result of \cite[Proposition III-1.6]{GHKLMS03}. See also \cite[\S 3.5]{K17}.)

Now, recall from \cite{K17} that $(A_+,+,\pprec)^{*}\simeq \TT(A)$, for any $\Ca$-algebra. Further, it has been shown in \cite{ERS11} that the point-wise order on $\TT(A)$ coincide the algebraic order. As a consequence, any cone morphism from $\TT(A)$ to $[0,\infty]$ is automatically order-preserving. Applying the Claim to the preCuntz semigroup $M:=(A_+,+,\pprec)$ gives us the equality $(A_+,+,\pprec)^{**}=\Lsc_{\mathcal{C}}(\TT(A))$.

Subsequently, by \cite[\S 3.5]{K17}, the canonical map $\hat{\iota}\colon \mathcal{RJ}(A)\rightarrow (A_+,+,\pprec)^{**}$ sending a round-ideal $J\mapsto \sup \{\overline{x} \mid x\in J\}$, where $\overline{x}\colon (A_+,+,\pprec)^{*}\rightarrow [0,\infty]$ is the evaluation map, is Scott-continuous. 

Let $a\in (A_+,+,\pprec)$. We know that $\tau(a)=\sup_{\epsilon>0} \tau((a-\epsilon)_+)$ for any $\tau$. (By lower-semicontinuity.) Arguing similarly as in the proof of \autoref{thm:commu}, we also know that for any $b\in a^{\pprec}$, there exists $\epsilon>0$ such that $b\pprec (a-\epsilon)_+\pprec a$. As a consequence, we explicitly compute that $\hat{\iota}(a^{\pprec})= \overline{a}$, for any $a\in A_+$.
Therefore, the canonical order-embedding $\iota\colon \mathcal{R}(A)\lhook\joinrel\rightarrow \Lsc_{\mathcal{C}}(\TT(A))$ agrees with the composition $\hat{\iota}\circ\mathfrak{r}\colon\mathcal{R}(A)\rightarrow (A_+,+,\pprec)^{**}$.

Finally, the answer after \cite[Question 4.4]{K17} states that $\hat{\iota}\colon \mathcal{RJ}(A)\rightarrow (A_+,+,\pprec)^{**}$ is an isomorphism, for any commutative $\Ca$-algebra $A$. It follows that $\iota\colon \mathcal{R}(A)\lhook\joinrel\rightarrow \Lsc_{\mathcal{C}}(\TT(A))$ is an isomorphism, for any commutative $\Ca$-algebra $A$.
\end{proof}

\begin{cor} Any $\rm AH$-algebra is tracially reflexive. 
\end{cor}
\section{A couple of criteria for Tracial Reflexiveness}
\subsection{A first criterion via the Cuntz semigroup} We prove that any separable $\Ca$-algebra with an algebraic lattice of ideals is tracially reflexive. The proof relies on Cuntz semigroup techniques, and more particularly, on results contained in \cite{MR23}. We refer the reader to \cite{GP24} for a detailed survey on the Cuntz semigroup. For the sake of readability, we recall a few known results about $\Cu$-semigroups, in the concrete setting of $\Ca$-algebras. 

Let $A$ be a $\Ca$-algebra and let $a,b\in A_+$. 
We say that $a$ \emph{Cuntz subequivalent} to $b$ in $A$, denoted $a \lesssim_{\Cu} b$, if for every $\epsilon > 0$,  there exists $ r\in A$ such that $\left\lVert rbr^*-a \right\rVert < \epsilon$. We antisymmetrize this subequivalence relation, to obtain a an equivalence relation $\sim_{\Cu}$ called the \emph{Cuntz equivalence relation}.
The \emph{Cuntz semigroup} of $A$ is 
  \[
  \Cu(A):=((A \otimes \mathcal{K})_{+}/\sim_{\Cu},+,\leq)
  \]
where the addition is canonically defined and the order is induced by $\precsim_{\Cu}$.
It has been shown that the functor $\Cu\colon \Ca\longrightarrow \Cu$ is well-defined and continuous. See \cite{CEI08, APT18}.

A \emph{functional on $\Cu(A)$} is a monoid morphism $\lambda\colon \Cu(A)\rightarrow [0,\infty]$, preserving suprema of increasing sequences. The set of functionals is denoted by $\FF(\Cu(A))$. Equipped with the point-wise operation, $\FF(\Cu(A))$ becomes an extended Choquet cone. Additionally, the contravariant functors $\FF\colon \Cu\rightarrow \mathcal{C}_{\rm Choq}$ and $\TT\colon \Ca\rightarrow \mathcal{C}_{\rm Choq}$ are naturally isomorphic. See \cite{ERS11, MR23} for more details. 

\emph{The realification of $\Cu(A)$}, denoted by $\Cu(A)_R$, is the smallest $\Cu$-cone of $\Lsc_\CC(\FF(\Cu(A)))$ containing $\{\hat{s}\mid s\in \Cu(A)\}$, where $\hat{s}\colon \FF(\Cu(A))\rightarrow [0,\infty]$ is the d-monoid morphism sending $\lambda\rightarrow \lambda(s)$. This is a $\Cu$-cone isomorphic to $\Cu(A\otimes \mathcal{W})$, where $\mathcal{W}$ denotes the Jacelon-Razak algebra. See \cite[Lemma 5.1.3]{R13}. 

\begin{thm}
Let $A$ be a separable $\Ca$-algebra be such that $\Cu(A)\otimes \{0,\infty\}$ is algebraic. 

Then $A$ is tracially reflexive.
\end{thm}

\begin{proof}
The proof essentially relies on the remarkable result contained in {\cite[Theorem 5.3]{MR23}}, stating that $S\simeq (\Lsc_\mathcal{C}(\FF(S))_\sigma$ for any $\Cu$-cone $S$ satisfying (O5) and (O6), and whose lattice of ideals is algebraic.

Fix $S:=\Cu(A\otimes \mathcal{W})$. As recalled above, $S\simeq \Cu(A)_R$ and hence, is a $\Cu$-cone. Being the Cuntz semigroup of a $\Ca$-algebra, $S$ automatically satisfies (O5) and (O6). See \cite[Proposition 5.1.1]{R13}. Additionally, it is known that $ \Lat(A)\simeq\Lat(A\otimes \mathcal{W})\simeq\Lat(S) \simeq \Cu(A)\otimes\{0,\infty\}$. (See \cite[\S 5]{APT18} and \cite[\S 7]{APT18}, for the second and last isomorphisms, respectively.)

We now apply \cite[Theorem 5.3]{MR23} to deduce that $S\simeq (\Lsc_\mathcal{C}(\FF(S)))_\sigma$. Under the separable assumption of $A$, we get that $\FF(S)$ is metrizable. Henceforth, $(\Lsc_\mathcal{C}(\FF(S)))_\sigma=\Lsc_\mathcal{C}(\FF(S))$. 
Gathering all the above, we obtain $\mathcal{R}(A)\simeq \LL(\TT(A))\simeq\Cu(A)_R\simeq\Lsc_\mathcal{C}(\FF(\Cu(A)_R))\simeq\Lsc_\mathcal{C}(\FF(\Cu(A)))\simeq \Lsc_\mathcal{C}(\TT(A))$. 

(We have successively used \autoref{cor:triangle} (ii), \cite[Theorem 3.2.1]{R13}, \cite[Theorem 5.3]{MR23}, \cite[Theorem 3.1.1]{R13} and \cite[Theorem 4.8]{ERS11}.)
\end{proof}

The result contained in \cite[Theorem 1.10]{PR07}, combined with the fact that $\Cu(A\otimes \mathcal{O}_\infty)\simeq \Cu(A)\otimes \{0,\infty\}$, can be used to show that a separable $\Ca$-algebra $A$ has topological dimension zero if and only if its lattice of ideals is algebraic. See \cite{NTV26}. As a result, we immediately obtain the following corollary.

\begin{cor} Any separable $\Ca$-algebra of topological dimension zero is tracially reflexive. These include the following separable $\Ca$-algebras.

\begin{itemize}
\item[(i)] Simple $\Ca$-algebras. E.g., Villadsen algebras and selfless $\Ca$-algebras.
\item[(i)] $\Ca$-algebras with finitely many ideals.
\item[(iii)] $\Ca$-algebras having the (weak) ideal property. E.g., Real rank zero $\Ca$-algebras.
\end{itemize}
\end{cor}

\subsection{A second criterion via dual pairs of cones}We next expose an abstract class of tracially reflexive $\Ca$-algebras. More particularly, we find a weaker condition on the lattice of closed two-sided ideals, which suffices to deduce tracial reflexiveness of the $\Ca$-algebra at hand. In view of setting our theorem, we first recall several preliminary results that we gather below. 

\begin{prg} \textbf{A handy picture of the cone of traces}.
Following \cite[\S 3.3]{APRT21} and \cite[\S 2.2]{MR23}, we may use the lattice of ideals of a $\Ca$-algebra $A$ to provide a useful decomposition of its cone of traces $\TT(A)$, via idempotents elements. We highlight that all the above can be done in the abstract context of $\Cu$-semigroups and their \emph{cones of functionals}.

By an \emph{idempotent} of a monoid $M$, we mean an element $m$ such that $2m=m$, and we denote the set of idempotents by $\EE(M)$. 

The set of idempotents in the cone of traces of a $\Ca$-algebra $A$ encodes its lattice of ideals. More particularly, there are order-isomorphisms 
\[
\begin{array}{ll}
(\Lat(A),\subseteq)\simeq (\EE(\TT(A)),\geq)\simeq (\EE(\Lsc_{\mathcal{C}}(\TT(A))),\leq)\\
\hspace{1cm}I\hspace{0,5cm}\longmapsto \hspace{0,5cm}\tau_I \hspace{1cm}  \longmapsto \hspace{0,5cm}f_I
\end{array}
\]
where $\tau_I$ takes value $0$ on $I$ and $\infty$ otherwise. Similarly, $f_I$ takes value $0$ whenever $\tau\leq \tau_I$ and $\infty$ otherwise. 

A fortiori, both $(\EE(\TT(A)),\geq)$ and $(\EE(\Lsc_{\mathcal{C}}(\TT(A))),\leq)$ are complete lattices.
\begin{dfn} Let $A$ be a $\Ca$-algebra. Let $\tau\in \TT(A)$ and let $f\in \Lsc_{\mathcal{C}}(\TT(A))$.
\begin{itemize}
\item[(i)] The \emph{support idempotent} of $\tau$ is defined by $\varepsilon(\tau):=\sup\{e\in \EE(\TT(A))\mid e\leq \tau\}$.
\item[(ii)] The \emph{support idempotent} of $f$ is defined by $\varepsilon(f):=\inf\{e\in \EE(\Lsc_\CC(\TT(A)))\mid f\leq e\}$. 
\item[(iii)] The \emph{support ideals} of $\tau$ and $f$ are defined to be the respective images of $\varepsilon(\tau)$ and $\varepsilon(f)$ via the above identifications.
\end{itemize}
\end{dfn}

We have the following computations.

\begin{prop}\label{prop:computeideals} Let $A$ be a $\Ca$-algebra. Let $\tau\in \TT(A)$ and let $f\in \Lsc_{\CC}(\TT(A))$. We compute that
\begin{itemize}
\item[(i)] $\varepsilon(\tau)=\inf\limits_{n} \frac{1}{n}\tau$ and $I_\tau$ is the closed two-sided ideal generated by $\{a\in A_+\mid \tau(a)<\infty\}$.
\item[(ii)]  $\varepsilon(f)=\sup\limits_{n} nf$ and $I_f$ is the closed two-sided ideal generated by $\{a\in A_+ \mid \overline{a}\leq \varepsilon(f)\}$. 

As a consequence, for any $e\in \EE(\Lsc_{\mathcal{C}}(\TT(A)))$, we have that $f(e)=\infty$ if and only if $e\nleq \varepsilon(f)$. 
\end{itemize}
\end{prop}

\begin{proof} The computation of $\varepsilon(\tau)$ is readily obtained and $\varepsilon(\tau)(a)= \infty$ if and only if $\tau(a)=\infty$. Similarly, the computation of $\varepsilon(f)$ is easily done and $\varepsilon(f)(\tau)= 0$ if and only if $f(\tau)=0$. 
\end{proof}

We can now describe the cone of traces of a $\Ca$-algebra $A$ as the disjoint union 
\[\TT(A)=\bigsqcup\limits_{I\in \Lat(A)} \TT_I(A)\] where $\TT_I(A):=\{\tau\in \TT(A) \mid \varepsilon(\tau)=\tau_I\}$. 

Note that $\TT_I(A)$ is a monoid (with neutral element $\tau_I$) closed under $\R_+$-multiplication.
\end{prg}

\begin{prg} \textbf{The Pedersen ideal and $\boldsymbol{C_\mathcal{C}(\TT(A))}$}. For specific ideals $I\in \Lat(A)$, we are able to describe the behavior of $\TT_I(A)$ and its continuous functions, via their Pedersen ideals. 

The \emph{Pedersen ideal} $\Ped(A)$ of a $\Ca$-algebra $A$ is the smallest dense ideal in $A$. Such an ideal is a priori \emph{not} closed, and always exists. See \cite[\S 5.6]{P79}. A useful characterization using the Cuntz semigroup is given by $\Ped(A)=\{a\in A_+ \mid [a]\in \Cu(A)_\ll\}$. See e.g. \cite{TT15}. 

For any closed two-sided ideal $I\in \Lat(A)$, the topology induced by $\tau_{iv}$ on $\TT_I(A)$ coincides with the point-wise convergence on the positive elements of the Pedersen ideal of $I$. Consequently, any net $(\tau_i)_i$ in $\TT_I(A)$ converging to $\tau$ for the interval topology, satisfies that $\lim_i \tau_i(a)=\tau(a)$, for any positive element $a\in \Ped(I)_+$. See \cite[Proposition 3.11 (i)]{ERS11}. In particular, any $a\in \Ped(I)_+$ induces a continuous cone morphism $\overline{a}_\mid\colon \TT_I(A)\rightarrow \R_+$. 

Lastly, we can find a \emph{compact base} of the (cancellative) topological cone $\TT_I(A)$ -i.e., a compact subset $T$ such that $0_{\TT_I(A)}\notin T$ and $\R_+T=\TT_I(A)$-, whenever the ideal $I$ satisfies one of the following conditions. 

\begin{prop}\label{prop:compactbase} Let $A$ be a $\Ca$-algebra and let $I$ be a closed two-sided ideal of $A$. The following statements are equivalent.
\begin{itemize}
\item[(i)] $I$ is singly-generated and is compact element of the complete lattice $\Lat(A)$. 
\item[(ii)] There exists a full element $e_I$ contained in $\Ped(I)$.
\item[(iii)] $\Prim(I)$ is a compact space.
\end{itemize}
Furthermore, they all imply
\begin{itemize}
\item[(iv)] $\TT_I(A)$ is a locally compact (cancellative) subcone of $\TT(A)$, with compact base $T_{e_I}(A)$.
\end{itemize}
\end{prop}

\begin{proof}
For $(i)\Rightarrow (ii)$, take $I\in\Lat(A)$ be a singly generated compact ideal. By the bijective correspondance $\Lat(A)\simeq\Lat(\Cu(A))$, we can find $x\in \Cu(A)$ such that $I=I_x$. Now, choose any $\ll$-increasing sequence $(x_n)_n$ in $\Cu(A)$ such that $\sup x_n=x$. Then $(I_{x_n})_n$ is a $\subseteq$-increasing sequence whose supremum is $I$. By the compactness of $I$, we deduce that there exists $n\in \N$ such that $I=I_{x_n}$. Lastly, choose $e_I\in A_+$ with $[e_I]=x_n$. It is easily observed that $[e_I]\in I_\ll$, that is, $e_I\in \Ped(I)_+$ and $e_I$ is full by construction.
 
For $(ii)\Rightarrow (i)$, take an increasing sequence $(I_n)_n$ in $\Lat(I)$ satisfying that $\sup I_n=I$. Since $[e_I]_{\Cu}\in \Cu(I)_\ll$, it follows that $e_I\in \bigcup I_n$ and hence, $e_I\in I_n$ for some $n\in \N$. As a result, $I_n=I$.
 
The equivalence between $(ii)$ and $(iii)$ is \cite[Proposition 3.1]{TT15}. The implication of $(ii)\Rightarrow (iv)$ is from \cite[Proposition 3.1]{TT15} together with Klee's theorem, stating that a cancellative topological Hausdorff cone is locally compact if and only if it has a compact base. See e.g. \cite[Theorem II.2.6]{A71}. 
\end{proof}

Let us denote the subset of \emph{continuous} maps of $\Lsc_\mathcal{C}(\TT(A))$ by $C_\mathcal{C}(\TT(A))$. (That is, the set of cone morphisms $f\colon\TT(A)\rightarrow [0,\infty]$ that are $(\tau_{iv},\tau_{\rm standard})$-continuous, where $[0,\infty]$ is equipped with the standard topology.) From the results obtained in \cite[\S 4]{MR23}, we know that for any $f\in C_\mathcal{C}(\TT(A))$, the support ideal $I_f$ is a compact element of $\Lat(A)$. Furthermore $f(\tau)\in (0,\infty)$, for any $\tau\in \TT_{I_f}(A)$ and $f(\tau)=\infty$ if and only if $I_\tau \nsupseteq I_f$. See \autoref{prop:computeideals}. 

Lastly, fix a compact ideal $I\in \Lat(A)$. Any continuous cone morphism $f\colon \TT_I(A)\rightarrow [0,\infty)$ such that $f^{-1}(\{0\})= \tau_{I}$, admits a unique continuous extension $\tilde{f}$, given by $\tilde{f}(\tau)=f(\tau+\tau_I)$ whenever $\tau+\tau_I\in \TT_I(A)$, and $\tilde{f}(\tau)=\infty$ otherwise. It can be checked that $\widetilde{f}$ belongs to $C_\mathcal{C}(\TT(A))$ and satisfies $I_f=I$. See the proof of \cite[Theorem 4.2]{MR23}.

We end this paragraph with the following proposition linking elements of Pedersen ideals of closed two sided ideals and elements of $\Lsc_\mathcal{C}(\TT(A))$ that are continuous.

\begin{prop}\label{prop:pepite} Let $A$ be a $\Ca$-algebra and let $a\in A_+$. Assume that $a \in \Ped(I_a)_+$. 

Then the restriction $\overline{a}_\mid\colon \TT_I(A)\rightarrow [0,\infty]$ is continuous, its image lies in $\R_+$, and we have $(\overline{a}_\mid)^{-1}(\{0\})=\tau_{I_a}$. 

Moreover, its unique continuous extension (as constructed above) coincides with $\overline{a}$.
\end{prop}

\begin{proof}
As explained in the above paragraph, \cite[Proposition 3.11 (i)]{ERS11} gives the continuity of $\overline{a}_\mid$, the belonging of $a$ to the Pedersen ideal of $I$ implies that its image lies in $\R_+$, while the fullness of $a$ in $I_a$ yields $\tau(a)>0$ for any $\tau\in \TT_I(A)\setminus \{\tau_I\}$. We are left to show that the unique continuous extension obtained from \cite[Theorem 4.2]{MR23} is in fact $\overline{a}$ itself. 

Let us write $f:=\overline{a}_\mid$. Following the proof-steps of the cited result, we continuously extend $f$ to $\widetilde{f}\colon \TT(A)\rightarrow [0,\infty]$, by sending $\tau\mapsto f(\tau+\tau_{I_a})$, whenever $\epsilon(\tau)\leq \tau_{I_a}$ (i.e., whenever $I_\tau\supseteq I_a$) and $\tau\mapsto \infty$ otherwise. It is now straightforward that $\widetilde{f}$ coincides with $\overline{a}$.
\end{proof}
\end{prg}

\begin{prg} \textbf{Dual pairs of cones and the weak Schr\"{o}der-Simpson theorem}. Let us briefly expose another work of Keimel based on \cite{K15}, where he presents a conceptual approach to investigate possible generalizations of the Schr\"{o}der-Simpson theorem outside the valuation powerdomain. We refer the reader to \cite{K15} and the references therein for details on the Schr\"{o}der-Simpson theorem, and we focus on the point-wise (weaker) version obtained by Keimel, in the broader context of dual pairs of cones, which is of interest to us. 

By a \emph{dual pair of cones}, we mean a triple $(C,D,\omega)$, where $C,D$ are cones and $\omega\colon C\times D\rightarrow [0,\infty]$ is bilinear, (i.e., a cone morphism in each variable) and non-singular (i.e., for any $d,d'\in D$ with $d\neq d'$, there exists $c\in C$ such that $\omega(c,d)\neq \omega(c,d'))$. We endow the cone $D$ with the \emph{weak upper} topology $\tau_\omega$, defined as the coarsest topology on $D$ rendering lower-semicontinuous any evaluation map $\ev_c\colon D\rightarrow [0,\infty]$, sending $d\mapsto \omega(c,d)$, for all $c\in C$. In this broader context, Keimel has been able to prove the following.

\begin{thm}[{\cite[Theorem 4.3]{K15}}] Let $(C,D,\omega)$ be a dual pair of cones. Any lower-semicontinuous function $f\colon D\rightarrow [0,\infty]$ (i.e., any $(\tau_\omega,\tau_{\rm Scott})$-continuous function) is the point-wise supremum of evaluation maps of elements in $C$.

More particularly, we have $f(d)=\sup\{\omega(c,d) \mid \ev_c\leq f\}$, for any $d\in D$.
\end{thm}

Let us transcribe the above theorem in our setting. For a $\Ca$-algebra, we consider the map $\omega_A\colon \mathcal{R}(A)\times \TT(A)\rightarrow [0,\infty]$, sending $(\overline{a},\tau)\mapsto\tau(a)$. It is readily checked that $(\mathcal{R}(A),\TT(A),\omega_A)$ is a well-defined dual pair of cones, and that the weak upper topology $\tau_\omega$ on $\TT(A)$ coincides the weak*-upper topology $\tau_{up}$ on $(A_+,+,\pprec)^*$ considered in \autoref{prg:Keimel}. As a consequence, we immediately obtained the following corollary.

\begin{cor}\label{cor:SS} Let $A$ be a $\Ca$-algebra. Any $f\in \Lsc_\mathcal{C}(\TT(A))$ is the point-wise supremum of evaluation maps of elements in $(A\otimes \mathcal{K})_+$. 

More particularly, we have $f(\tau)=\sup_{a\in (A\otimes \mathcal{K})_+}\{\tau(a) \mid \overline{a}\leq f \}$, for any $\tau\in \TT(A)$.
\end{cor}

\begin{proof} As already explained in the proof of \autoref{thm:commutative}, we have $\Lsc_\mathcal{C}((\TT(A),\tau_{iv})=\Lsc_\mathcal{C}((\TT(A),\tau_{up})$, which now also agrees with $\Lsc_\mathcal{C}((\TT(A),\tau_{\omega})$.
\end{proof}
\end{prg}

We have now all the theory needed to exhibit our second criterion for tracial reflexiveness. Let us start with a lemma.

\begin{thm}
Let $A$ be a separable $\Ca$-algebra. Assume that for any $I\in \Lat(A)$, we have that $\Ped(A)_+\cap I\subseteq \Ped(I)_+$.

Then $A$ is tracially reflexive.
\end{thm}

\begin{proof}
Let $f\in \Lsc_\mathcal{C}(\TT(A))$. We aim to find an element $a\in A\otimes\mathcal{K}_+$ such that $\overline{a}=f$. We argue similarly as in the proof of \cite[Theorem 4.4]{MR23}, and we proceed in several steps. Consider the following sets.
\[
\begin{array}{ll} 
P_f:=\{ a\in \Ped(A\otimes\mathcal{K})_+\mid \overline{a}\leq (1-\epsilon)f \text{ for some }\epsilon >0\}.\\
H_f:=\{ h\in C_\mathcal{C}(\TT(A)) \mid h\leq (1-\epsilon)f \text{ for some }\epsilon >0\}.\\
R_f:=\{ a\in A\otimes\mathcal{K}_+\mid \overline{a}\leq f\}.
\end{array}
\]

We first observe that $P_f\subseteq H_f\subseteq R_f$. The first inclusion follows from \autoref{prop:pepite} together with the hypothesis applied to $I_a$. (That is, $\Ped(A)_+\cap I_a\subseteq\Ped(I_a)_+$.) The second inclusion follows from the fact that $C_{\mathcal{C}}(\TT(A))\subseteq \{\overline{a} \mid a\in A\otimes \mathcal{K}_+\}$. (See \autoref{cor:triangle}.)

Now, let $\tau\in \TT(A)$. It follows from the above inclusions that \[\sup_{a\in P_f}\{\tau(a)\}\leq \sup_{h\in H_f}\{h(\tau)\}\leq \sup_{a\in R_f}\{\tau(a)\}\leq f(\tau).\] 

We shall prove that all the above terms are in fact equal. By the weak Schr\"{o}der-Simpson theorem recalled in \autoref{cor:SS}, we know that $\sup_{a\in R_f}\{\tau(a)\}= f(\tau)$. It now suffices to show that $\sup_{a\in R_f}\{\tau(a)\}\leq \sup_{a\in P_f}\{\tau(a)\}$ to get  the desired equalities. Let $a\in R_f$. Without loss of generality, we can assume that $\overline{a}\leq (1-\epsilon)f$, for some $\epsilon>0$. Further, the lower-semicontinuity of $\tau$ implies that $\tau(a)=\sup_{\delta>0} \tau((a-\delta)_+)$. Since $(a-\delta)_+$ belongs to $\Ped(A\otimes\mathcal{K})_+$ for any $\delta>0$, it results that $\tau(a)\leq \sup_{a\in P_f}\{\tau(a)\}$ and hence, $\sup_{a\in R_f}\{\tau(a)\}=\sup_{a\in P_f}\{\tau(a)\}$. A fortiori, we conclude that \[\sup_{h\in H_f}\{h(\tau)\}=f(\tau)\] for any $\tau\in \TT(A)$.

Subsequently, it is shown in (the proof of) \cite[Theorem 4.4]{MR23} that $H_f$ is an upward-directed set. Also, it is a standard fact that any upward-directed subset of a countably-based $\Cu$-semigroup has a supremum. (See e.g. \cite[Lemma 2.6]{APRT22}.) Since $A$ is separable, $\mathcal{R}(A)\simeq \Cu(A)\otimes [0,\infty]$ is a countably-based $\Cu$-semigroup. Henceforth, $H_f$ has a supremum $h_f$ in $\mathcal{R}(A)$, which straightforwardly satisfies that $h_f(\tau)=\sup_{h\in H_f}\{h(\tau)\}$. The proof follows after noticing that $h_l=\overline{a}$, for some $a\in A\otimes\mathcal{K}_+$.
\end{proof}

\subsection{Closing Remarks} 
\begin{prg}\textbf{Conjecture for stable rank one $\Ca$-algebras}. The second criteria for tracial reflexiveness can be detected in the Cuntz semigroup. More particularly, for any ideal $I$ in a $\Ca$-algebra $A$, the condition $\Ped(A)_+\cap I\subseteq \Ped(I)_+$ is met, whenever we have $\Cu(A)_\ll \cap \Cu(I)=\Cu(I)_\ll$. Although it is always true that $\Cu(A)_\ll\cap \Cu(I)\supseteq\Cu(I)_\ll$, the reverse inclusion may fail. However, in the stable rank case, it has been shown that the Cuntz semigroup of $A$ satisfies nice properties, such as the \emph{Riesz interpolation property} and \emph{stably finiteness} -i.e., $\Cu(A)_\ll\subseteq \Cu(A)_{\rm finite}$-.  We hint that a combination of the aforementioned properties could provide a proof that $\Cu(A)_\ll\cap \Cu(I)=\Cu(I)_\ll$, for any ideal of a stable rank one $\Ca$-algebra, thus yielding a proof of the following conjecture.\\

(\textbf{Conjecture}) Any stable rank one $\Ca$-algebra is tracially reflexive.
\end{prg} 

\begin{prg}\textbf{Link with tracial orderedness}. 
Does tracial reflexiveness imply tracial orderedness for $\Ca$-algebras with a non-empty tracial state space? 
We sketch an idea for a potential proof, in the specific case where $A\ll A$ in $\Lat(A)$. (E.g., $A$ is unital or simple.)

Let $A$ be a tracially relfexive $\Ca$-algebra such that $\TT_1(A)\neq \emptyset$.
We identify $\TT_1(A)$ with a subset of $(\TT(A),\tau_{iv})$ and we first show that the induced topology agrees with the weak*-topology. Endowed the weak*-topology, it is well known that $\TT_1(A)$ is a compact (convex) Hausdorff space. Endowed with the relative (interval) topology, $\TT_1(A)$ is also Hausdorff. Now, consider the cone $\TT_A(A)$ of densely defined traces. Observe that $\TT_1(A)\subseteq \TT_A(A)\subseteq \TT(A)$. 
By \cite[Proposition 3.11 (i)]{ERS11}, we know that the topology induced by $\tau_{iv}$ on $\TT_A(A)$ coincides with the topology of point-wise convergence on the positive elements of the Pedersen ideal of $A$. We now see that the topology induced by $\tau_{iv}$ on $\TT_1(A)$ is coarser than the weak*-topology. As a result, both topologies on $\TT_1(A)$ are Hausdorff and compact. Finally, since they compare on the lattice of topologies, they must coincide.

 Now let $f\in \Aff\TT_1(A)_+$ be a (weak*)-continuous map. Then $f$ is a $\tau_{iv}$-continuous map. We first extend it to a $\tau_{iv}$-continuous map on $\TT_A(A)$, via any compact base $\TT_e(A)\simeq \TT_1(A)$ obtained from \autoref{prop:compactbase}. Then on, we use the unique continuous extension recalled before \autoref{prop:pepite}, to obtain a $\tau_{iv}$-continuous map $\tilde{f}$ on $\TT(A)$, extending $f$. 
 
 Lastly, since $A$ is tracially reflexive, there exists $a\in (A\otimes\mathcal{K})_+$ such that $\overline{a}=\tilde{f}$. We are only left to prove that $a$ can be taken in $A_+$ to be able to conclude.
\end{prg}

\begin{prg}\textbf{Classification of *-homomorphisms between $\mathcal{W}$-stable $\Ca$-algebras}. L. Robert has conjectured that the (scaled) cone of lower-semicontinuous traces $\TT\colon \Ca\rightarrow \mathcal{C}_{\rm Choq}$ classifies separable nuclear $\Ca$-algebras, up to $\mathcal{W}$-tensorisation. That is, for any separable nuclear $\Ca$-algebras $A$ and $B$, we would have $\TT(A)\simeq \TT(B)$ if and only if $A\otimes \mathcal{W}\simeq B\otimes \mathcal{W}$. 

(This version of the conjecture would be possible only under the mild condition of preserving a scaling element. Otherwise, the conjecture should be stated up to \emph{stable} isomorphism.)

For the class of tracially reflexive $\Ca$-algebras, we see that the functors $\TT$, $\Lsc_\mathcal{C}\circ\TT$ and $\Cu$ can be recovered from one another. (See e.g. \cite[\S 5.3]{C21a}, for more on the notion of {recovering functors}.) More precisely, it follows from \cite{ERS11} that $\Cu$ recovers $\TT$, and from construction that $\TT$ recovers $\Lsc_\mathcal{C}\circ \TT$. Also, for any $\mathcal{W}$-stable tracially reflexive $\Ca$-algebra, we have that $\Lsc_\mathcal{C}(\TT(A))\simeq \Cu(A)\otimes [0,\infty]\simeq \Cu(A)$. As a result, having classification of these $\Ca$-algebras and their *-homomorphisms in mind, these invariants contain the same information. 

As a consequence, combining the results contained in \cite{R12} and \cite{S12} with the above, we deduce that the (scaled) Cuntz semigroup (or equivalently, either of the scaled functors $\TT$ or $\Lsc_\mathcal{C}\circ\TT$) classifies *-homomorphisms from $A\otimes \mathcal{W}\rightarrow B\otimes \mathcal{W}$, where $A\in \AF$ and $B\in \AH$. (See e.g. \cite{C25}, for more on the notion of \emph{classifying *-homomorphisms}.)

Lastly, tracial reflexiveness allows for a metric on the set morphisms. More particularly, for any *-homomorphisms $\phi,\psi\colon A\rightarrow B$ between $\mathcal{W}$-stable tracially reflexive $\Ca$-algebras, we know that $\Lsc_\mathcal{C}(\TT(A))$ is a $\Cu$-cone. Therefore, we are able to compare $\alpha:=\Lsc_\mathcal{C}(\TT(\phi))$ and $\beta:=\Lsc_\mathcal{C}(\TT(\psi))$ as follows.
\[d(\alpha,\beta):= \inf \Bigl\{ r>0\mid \forall  U\in\mathcal{O}(\TT(A)),\, \alpha(\mymathbb{1}_{U})\leq\beta(\mymathbb{1}_{U_{r}})  \text{ and }  \beta(\mymathbb{1}_{U})\leq\alpha(\mymathbb{1}_{U_{r}})\Bigr\}\]
where $\mathcal{O}(\TT(A)):=\{\text{Open sets of }\TT(A)\}$ and $U_r:=\underset{x\in U}{\cup}B_r(x)$. If the infimum does not exist, we set the value to $\infty$. 

(See e.g. \cite{C24} for an explicit use of such metric. See also \cite{MR23} for an explicit sub-basis of open sets for the interval topology on $\TT(A)$.) This would allow to not only classify *-homomorphisms between specific  $\mathcal{W}$-stable tracially reflexive $\Ca$-algebras, but also compute how far they are from each other.
\end{prg}


\begin{thebibliography}{99}
\bibitem{A71}
M.~Alfsen.
\newblock Compact convex sets and boundary integrals.
\newblock Springer-Verlag, New York, 1971. Ergebnisse
der Mathematik und ihrer Grenzgebiete, Band 57.

\bibitem{APRT21}
R.~Antoine, F.~Perera, L.~Robert, and H.~Thiel.
\newblock Edwards' condition for quasitraces on {$\mathrm{C}^*$}-algebras.
\newblock {\em Proc. Roy. Soc. Edinburgh Sect. A}, 151(2):525--547, 2021.

\bibitem{APRT22}
R.~Antoine, F.~Perera, L.~Robert, and H.~Thiel.
\newblock $\mathrm{C}^*$-algebras of stable rank one and their {C}untz
  semigroups.
\newblock {\em Duke Math. J.}, 171(1):33--99, 2022.




\bibitem{APT18}
R.~Antoine, F.~Perera, and H.~Thiel.
\newblock Tensor products and regularity properties of {C}untz semigroups.
\newblock {\em Mem. Amer. Math. Soc.}, 251(1199):viii+191, 2018.
%
%

\bibitem{C21a}
L.~Cantier.
\newblock A unitary {C}untz semigroup for {$\mathrm{C}^*$}-algebras of stable
  rank one.
\newblock {\em J. Funct. Anal.}, 281(9):109175, 2021.





\bibitem{C24} L.~Cantier. 
\newblock On the Nielsen-Thomsen sequence.
\newblock Preprint. arXiv:2412.11975, 2024

\bibitem{C25}
L.~Cantier.
\newblock Towards a classification of unitary elements of
  $\mathrm{C}^*$-algebras.
\newblock {\em Int. Math. Res. Not. IMRN}, 7, 1--19, 2025.


\bibitem{CGSTW21}
J. R. Carri\'on, J. Gabe, C. Schafhauser, A. Tikuisis and S. White
\newblock Classifying *-homomorphisms I: Unital simple nuclear $\mathrm{C}^*$-algebras
\newblock Preprint. arXiv:2307.06480, 2023.




\bibitem{CEI08}
K.~T. Coward, G.~A. Elliott, and C.~Ivanescu.
\newblock The {C}untz semigroup as an invariant for {$\mathrm{C}^*$}-algebras.
\newblock {\em J.\ Reine Angew.\ Math.}, 623:161--193, 2008.

\bibitem{C78}
J.~Cuntz.
\newblock Dimension functions on simple {$\mathrm{C}^*$}-algebras.
\newblock {\em Math. Ann.}, 233(2):145--153, 1978.

\bibitem{CP79}
J.~Cuntz and G. K.~Pedersen
\newblock Equivalence and traces on {$\mathrm{C}^*$}-algebras.
\newblock {\em J. Funct. Anal.}, 33(2):135--164, 1979.


\bibitem{EGLN21} G. A. Elliott, G. Gong, H. Lin and Z. Niu. 
\newblock On the classification of simple amenable C*-algebras with finite decomposition rank, II.
\newblock {\em J. Noncommut. Geom.}. In press. \url{10.4171/JNCG/560}, 2024.


\bibitem{ERS11}
G.~A. Elliott, L.~Robert, and L.~Santiago.
\newblock The cone of lower-semicontinuous traces on a
  {$\mathrm{C}^*$}-algebra.
\newblock {\em Amer. J. Math.}, 133(4):969--1005, 2011.

\bibitem{F92}
A. Fiech. 
\newblock Colimits in the category CPO.
\newblock Technical Report, Kansas State University, 1992.

\bibitem{GM26}
J. Gabe and A. Miller. 
\newblock Self-adjoint traces on the Pedersen ideal of $\mathrm{C}^*$-algebras.
\newblock {\em Publ. Mat.}, 2026. To appear

\bibitem{GP24}
E.~Gardella and F.~Perera.
\newblock The modern theory of {C}untz semigroups of $\mathrm{C}^*$-algebras.
\newblock {\em EMS Surv. Math. Sci.}, 2024.

\bibitem{GHKLMS03}
G.~Gierz, K. H.~Hofmann, K.~Keimel, J. D.~Lawson, M.~Mislove, and D. S.~Scott
\newblock Continuous lattices and domains.
\newblock Encyclopedia of Mathematics and its Applications 93, Cambridge University Press, Cambridge, 2003.

\bibitem{GLN1}
G. Gong, H. Lin, and Z. Niu.
\newblock A classification of finite simple amenable $\mathcal{Z}$-stable C*-algebras. I: C*-algebras with generalized tracial rank one.
\newblock {\em C. R. Math. Acad. Sci. Soc. R. Canada}, 42, pp. 63-450, 2020.

\bibitem{GLN2} 
G. Gong, H. Lin and Z. Niu.
\newblock A classification of finite simple amenable $\mathcal{Z}$-stable C*-algebras, II: C*-algebras with rational generalized tracial rank one.
\newblock {\em C. R. Math. Acad. Sci. Soc. R. Canada}, 42, pp. 451-539, 2020.

\bibitem{J25}
B.~Jacelon.
\newblock Quantum metric Choquet simplices.
\newblock {\em Publ. Math.}, 2025. To appear.

\bibitem{K15}
K.~Keimel.
\newblock Weak upper topologies and duality for cones.
\newblock {\em Log. Methods Comput. Sci.}, Vol. 11(3:21), pp. 1--14, 2015.

\bibitem{K17}
K.~Keimel.
\newblock The Cuntz semigroup and domain theory.
\newblock {\em Soft Comput.}, 21:2485--2502, 2017.

\bibitem{MR23}
M.~Moodie and L.~Robert.
\newblock Cones of traces arising from AF $\mathrm{C}^*$-algebras.
\newblock {\em Doc. Math.}, 28(6):1279--1321, 2023.

\bibitem{NTV26} 
\newblock P. W. Ng, H. Thiel, E. Vilalta. 
\newblock The Global Glimm Property for C*-algebras of topological dimension zero. Bull. London Math. Soc, 2026. (To appear)


\bibitem{PR07}
C. Pasnicu and M. R{\o}rdam. 
\newblock Purely infinite $\mathrm{C}^*$-algebras of real rank zero.
\newblock J. Reine Angew. Math. 613, 51--73, 2007.

\bibitem{P79}
G. K. Pedersen. 
\newblock $\mathrm{C}^*$-algebras and their automorphism groups.
\newblock Academic Press, London, 1979.

\bibitem{R09}
L.~Robert.
\newblock On the Comparison of Positive Elements of a $\mathrm{C}^*$-algebra by lower-semicontinuous Traces.
\newblock {\em Indiana Univ. Math. J.}, 58(6):2509--2515, 2009.

\bibitem{R12}
L.~Robert.
\newblock Classification of inductive limits of 1-dimensional {NCCW} complexes.
\newblock {\em Adv. Math.}, 231(5):2802--2836, 2012.

\bibitem{R13}
L.~Robert.
\newblock The cone of functionals on the {C}untz semigroup.
\newblock {\em Math. Scand.}, 113(2):161--186, 2013.


\bibitem{RS10}
L.~Robert and L.~Santiago.
\newblock Classification of {$C\sp \ast$}-homomorphisms from {$C\sb 0(0,1]$} to
  a {$\mathrm{C}^*$}-algebra.
\newblock {\em J. Funct. Anal.}, 258(3):869--892, 2010.

\bibitem{S12}
L. Santiago.
\newblock Reduction of the dimension of nuclear $\Ca$-algebras.
\newblock Preprint. arXiv:1211.7159, 2012.


\bibitem{TT15}
A. Tikuisis and A. Toms.
\newblock On the structure of Cuntz semigroups in (possibly) nonunital $\mathrm{C}^*$-algebras.
\newblock {\em Canad. Math. Bull.} 58, 402--414, 2015.

\bibitem{T17} H.~Thiel.
\newblock The Cuntz Semigroup.
\newblock Lecture notes from a course at the University of M\"unster, winter semester 2016-17. Available at: \url{https://ivv5hpp.uni-muenster.de/u/h_thie08/teaching/CuScript.pdf}.


\bibitem{W18} W. Winter.
\newblock Structure of nuclear $\Ca$-algebras: from quasidiagonality to classification and back again. 
\newblock {\em Proceedings of the International Congress of Mathematicians-Rio de Janeiro}, Vol. III. Invited lectures, pp 1801-1823, 2018. 
\end{thebibliography}
\end{document}